\documentclass[12pt]{amsart}
\usepackage{amsmath}
\usepackage{amssymb}
\usepackage{amsthm}
\usepackage{verbatim}
\usepackage{color}

\newtheorem{theorem}{Theorem}[section]
\newtheorem{lemma}[theorem]{Lemma}
\newtheorem{corollary}[theorem]{Corollary}
\newtheorem{proposition}[theorem]{Proposition}

\newtheorem{remark}[theorem]{Remark}

\begin{document}
\title[Isometries and Hermitian operators on  $\mathcal{B}_0(\triangle, E)$]{Isometries on the vector valued little Bloch space}

\author{Fernanda Botelho}
\address{Department of Mathematical Sciences, The University of Memphis, Memphis, TN 38152, USA}
\email{mbotelho@memphis.edu}

\author{James Jamison}
\address{Department of Mathematical Sciences, The University of Memphis, Memphis, TN 38152, USA}
\email{jjamison@memphis.edu}

\date{\today}

\begin{abstract}
In this paper we describe the surjective linear isometries on a vector valued little Bloch space with range space a strictly convex and smooth complex Banach space. We also describe the hermitian operators and the generalized bi-circular projections supported by these spaces.
\end{abstract}
\maketitle

\section{Introduction}
The type of linear surjective isometries supported by a given Banach space depends largely on the geometric properties of the  space. Often,  these operators  are described from their induced actions on the set of extreme points of the unit ball of the dual space. In addition of being a class of operators of great intrinsic interest, linear surjective isometries play a crucial role in the definition of other important classes of operators such as the hermitian operators and the generalized bi-circular projections.
In this paper, we  give a characterization of the surjective isometries on  vector valued little Bloch spaces and then derive the form of the hermitian operators and the generalized bi-circular projections.

The little Bloch space consists of all analytic functions $f$ defined on the open unit disc, $\triangle =\{ z \in \mathbb{C}: \, |z|<1\},$ with values in a smooth strictly convex  Banach space $E$ with norm $\| \cdot \|_E$, which satisfy the condition
\[ \lim_{|z|\rightarrow \, 1} \, (1-|z|^2) \|f'(z)\|_E=0.\]
This space with the norm $\|f\|_{\mathcal{B}}=\|f(0)\|_E+sup_{z\in \triangle} (1-|z|^2) \|f'(z)\|_E$ is a Banach space and will be  denoted by $\mathcal{B}_*(\triangle, \, E).$ Towards the characterization of the isometries on this setting, we start by considering   surjective isometries on $\mathcal{B}_0(\triangle, \, E),$
 the  subspace consisting
 of all functions in $\mathcal{B}_*(\triangle, \, E)$
  vanishing  at zero. Since $\mathcal{B}_*(\triangle, \, E)$  is isometrically isomorphic to $\mathcal{B}_0(\triangle, \, E) \oplus_1 E,$ and  $E$ does not support $L_1$-projections (see \cite{bg}), it follows that  $\mathcal{B}_0(\triangle, \, E)$ is also $L_1$-projection free. Hence  isometries on $\mathcal{B}_*(\triangle, \, E)$ admit a natural decomposition into an isometry on $\mathcal{B}_0(\triangle, \, E)$ and an isometry on $E$.

In order to derive a representation for a surjective isometry on $\mathcal{B}_0(\triangle, \, E)$, we  define an embedding of  $\mathcal{B}_0(\triangle, \, E)$ onto a closed subspace of  $\mathcal{C}_0(\triangle, E)$.  A  result due to Brosowski and Deutsch  describes the form for  the   extreme points of the unit dual ball of a subspace of $\mathcal{B}_0(\triangle, \, E),$ see \cite[Corollary 2.3.6, p. 33]{fj}.  Therefore the action of the adjoint operator on the extreme points provides a scheme for the representation of  surjective linear isometries.

It was shown by Vidav in \cite{vi, vi1} that hermitian operators are essentially the generators of strongly
 continuous one parameter groups of surjective isometries.  The knowledge of the surjective isometries defines naturally a class of operators containing  the hermitian operators. In particular, we will show that for $\mathcal{B}_0(\triangle, E)$ this implies that bounded hermitian operators are in a one-to-one correspondence to the bounded hermitian operators of the range space.   Another class of operators considered here and directly linked to surjective isometries are the generalized bi-circular projections (gbp), introduced in \cite{fosner}. These projections have been studied and characterized in a variety of spaces. In most known cases,  gbps can be expressed as the average of the identity with an isometric reflection, see for example \cite{bo_ja,bo_ja1,king} and also \cite{sz}. In the  last section of this paper we extend this representation to gbps on  this new collection of spaces.

\section{Extreme points of $\mathcal{B}_0(\triangle, \, E)^*_1$}

We consider the following  embedding of $\mathcal{B}_0(\triangle, \, E)$  into $\mathcal{C}_0(\triangle, \, E)$
\begin{equation*} \begin{array}{lccl}  \Phi:&\mathcal{B}_0(\triangle, \, E) & \rightarrow & \mathcal{C}_0(\triangle, \, E) \\
                      & f  &\rightarrow &  F=\Phi (f):  \triangle \rightarrow  E,\end{array}\end{equation*} given by  $\Phi (f)(z)=(1-|z|^2) f'(z).$
 The map $\Phi$  is a linear isometry onto a closed subspace  of $\mathcal{C}_0(\triangle, \, E)$, denoted by $\mathcal{Y}$.
 We recall that $\mathcal{C}_0(\triangle, \, E)$ is the set of all $E$- valued continuous functions defined on $\triangle$ such that $\lim_{|z| \rightarrow 1}F(z)=0$.

A result due to Brosowski and Deutsch (see \cite{fj}) implies that   extreme points of the unit ball of the dual space of $\mathcal{Y}$ are functionals  of the form $e^{\star} \delta_{z},$ with $e^{\star} \in E_1^*$, $z\in \triangle$ and $\delta_z :\mathcal{B}_0(\triangle, \, E)\, \rightarrow \, E$  the evaluation map  $\delta_z(f)=f(z).$

We  now show that all such functionals are indeed extreme points of $\mathcal{Y}^*_1$, we denote the set of extreme points  by $ext(\mathcal{Y}_1^*).$

\begin{lemma}
 $ext(\mathcal{Y}^*_1)=\{ e^*\delta_z : e\in E_1, \,\,\mbox{and}\,\,\, z \in \triangle\}.$
\end{lemma}
\begin{proof}
We refer the reader to Corollary 2.3.6 in \cite{fj}, due to Brosowski and Deutsch,  which states  that $ext(\mathcal{Y}^*_1)\,\subset \, \{ e^*\delta_z : e^*\in ext(E_1), \,\,\mbox{and}\,\,\, z \in \triangle\}.$
We assume that there exists  a functional of the form $e^*\delta_z$, $\varphi_1$ and $\varphi_2$ in $\mathcal{Y}^*_1$, such that
\begin{equation} \label{extreme_points}  e^*\delta_{z}= \frac{\varphi_1+ \varphi_2}{2}.\end{equation}
Since $\mathcal{Y}$ is a closed subspace of $\mathcal{C}_0(\triangle, E)$, the Hahn-Banach Theorem implies the existence  extensions of $\varphi_1$ and $\varphi_2$, to  $\mathcal{C}_0(\triangle, E).$  These functionals are written as
\[ \tilde{\varphi_1}(F) = \int_{\triangle} \, F d \nu^* \,\,\,\, \mbox{and} \,\,\, \tilde{\varphi_2}(F) = \int_{\triangle} \, F d \mu^* ,\]
with  $\nu^*$ and $\mu^*$ representing  regular vector valued Borel measures on $\triangle$ with values on $E^*$.

Given $z_0 \in \triangle$, we consider the  function in $\mathcal{B}_0(\triangle, E)$
\[ f_0(z)= \frac{(1-|z_0|^2)z}{1-\overline{z_0}z}\,e.\]  Furthermore $\sup_{|z|<1}(1-|z|^2)\|f_0'(z)\|=(1-|z_0|^2) \|f_0'(z_0)\|$ and, for all $z\in \triangle\setminus \{z_0\}$,
\[ (1-|z|^2)\|f_0'(z)\|<(1-|z_0|^2)\|f_0'(z_0)\|=1.\]
We apply (\ref{extreme_points}) to the function $F_0(z)= (1-|z|^2) f_0'(z)$  to conclude that $\varphi_1(F_0)=\varphi_2(F_0)=1.$ If $|\nu^*|(\triangle \setminus \{z_0\})>0,$ then there exists a compact subset $K$ of $\triangle \setminus \{z_0\}$ such that $|\nu^*|(K)>0$. Clearly
\[ \sup_{z \in K} \|F_0(z)\|= \sup_{z \in K} (1-|z|^2)\|f_0'(z)\|=\alpha <1.\]
Hence
\begin{align*}
1=\tilde{\varphi_1}(F_0)&=|\int_{\triangle} F_0 d \nu^*| =\left |\int_{\{z_0\}} F_0  d \nu^*+ \int_K F_0  d \nu^*+\int_{(\triangle\setminus \{z_0\})\setminus K} F_0  d \nu^*\right|\\
& \leq |\nu^*|(\{z_0\}) +\alpha |\nu^*|(K) +|\nu^*|((\triangle\setminus \{z_0\})\setminus K) \\
& <|\nu^*|(\triangle) =1.
\end{align*}
This leads to an absurd and shows that $|\nu^*| (\triangle \setminus \{z_0\})=0$ and $\nu^* (\triangle \setminus \{z_0\})=0$. This  also implies that $\nu^*\{z_0\}$ is a norm one functional. A similar reasoning applies to $\mu^*$.
Given $F \in \mathcal{Y}$,  we have
\begin{align*}
e^* \delta_{z_0} (F) = &(1-|z_0|^2) e^* (f'(z_0))=\frac{\tilde{\varphi_1}(F)+\tilde{\varphi_2}(F)}{2} \\
 =& \frac{1}{2} \left( \int_{\{z_0\}} F d\nu^* + \int_{\{z_0\}} F d\mu^*\right)\\
 =& \frac{1}{2} \left[ \nu^*(z_0) (1-|z_0|^2) f'(z_0)+ \mu^* (z_0) (1-|z_0|^2) f'(z_0)\right].
\end{align*}  Therefore
\[ e^* (f'(z_0))= \frac{\nu^*(z_0) (f'(z_0))+\mu^*(z_0) (f'(z_0))}{2}.\] The smoothness of $E_1^*$ and since $e^*$ is an extreme point of $E_1^*$, we have that $\nu^*=\mu^*$ and $\varphi_1=\varphi_2.$ This  completes the proof.
\end{proof}
The next corollary gives a description of the extreme points of $\mathcal{B}_0(\triangle , E)^*_1.$
\begin{corollary} The set of extreme points $ext(\mathcal{B}_0(\triangle , E)^*_1)$ is equal to the set of functionals
$\tau:\mathcal{B}_0(\triangle , E)\rightarrow \mathbb{C}$ of the form $ \tau (f) = e^*(\Phi (f)(z))$, with $ z \in \triangle $ and  $ e^* \in ext(E^*_1).$
\end{corollary}
\begin{proof}
The isometry $\Phi$ induces  $\Phi^*: \mathcal{Y}^* \rightarrow \mathcal{B}_0(\triangle , E)^*$. This isometry defines a bijection between the corresponding sets of extreme points, consequently we have that $\Phi^* (e^* \delta_z) \in ext(\mathcal{B}_0(\triangle , E)^*_1),$ with $e^* \delta_z\in ext(\mathcal{Y}^*_1).$ This functional is defined as follows:
\[ \Phi^* (e^* \delta_z) (f) =  e^*(\Phi(f)(z)).\]
This completes the proof.
\end{proof}
\begin{remark}\label{l1_sum}
We observe that the function $f \rightarrow (f(0), f-f(0))$  defines  a surjective  isometry from $\mathcal{B}_*(\triangle, \, E)$ onto $E\bigoplus_1 \mathcal{B}_0(\triangle, \, E).$

It is well known (cf. \cite{fj}) that $ext(\mathcal{B}_*(\triangle, \, E)^*_1)= ext(E_1^*\, \bigoplus_{\infty}\, (\mathcal{B}_0(\triangle, \, E)^*_1).$ Therefore $ext(\mathcal{B}_*(\triangle, \, E)^*_1)=\{ (v^*,\tau): \, v^*\in E_1^*, \,\, \tau \in ext(\mathcal{B}_0(\triangle , E)^*_1) \,\,\mbox{and} \,\, (v^*,\tau)(f)= v^*(f(0))+ \tau(f)\}$
\end{remark}
\section{A Characterization of the surjective isometries on $\mathcal{B}_0(\triangle , E)$}
In this section we show that surjective linear isometries on $\mathcal{B}_0(\triangle , E)$ are translations of  weighted composition operators.

We consider a surjective linear isometry $T: \mathcal{B}_0(\triangle, E) \rightarrow \mathcal{B}_0(\triangle, E)$ and define $S: \mathcal{Y} \rightarrow \mathcal{Y}$ such that  $S \circ \Phi= \Phi \circ T.$ Hence $S^* : \mathcal{Y}^* \rightarrow \mathcal{Y}^*$ induces a permutation of $ext (\mathcal{Y}^*_1)$. Therefore, for every $u^* \in E_1^*$ and $z \in \triangle$, there exist $v^* \in E_1$ and $w \in \triangle$ such that
\[ S^* ( u^* \delta_{z} ) = v^* \delta_{w},\]
equivalently we write
\begin{equation} \label{meq}
(1-|z|^2) u^* (Tf)'(z)= (1-|w|^2) v^*(f'(w)), \,\,\, \mbox{for every } f \in \mathcal{B}_0(\triangle, E).
\end{equation}

Conceivably  $v^*$ and $w$  depend on the choice of $u^*$ and $z$, this determines the following  two maps:
$$\begin{array}{rcll}\sigma :& \triangle \times E_1^*& \rightarrow &\triangle  \\ & (z,u^*)& \rightarrow & w \end{array} \,\,\,\, \mbox{and} \,\,\, \begin{array}{rcll}\Gamma :& \triangle \times E_1^*& \rightarrow & E_1^* \\ & (z,u^*)& \rightarrow &  v^*. \end{array}$$
In the next two lemmas, we show that $\sigma$ is independent of the second coordinate and $\Gamma$ is independent of the first.

\begin{lemma} Let $z_0 \in \triangle$ and $u_0^* \in E_1^*.$ Then  $\sigma$ restricted to the set  $\{ (z_0,u^*): u^* \in E_1^*\}$ is constant and it induces a disc automorphism, also denoted by $\sigma$, defined by $\sigma (z)= \sigma (z, u_0^*)$.
\end{lemma}
\begin{proof}
 We consider two distinct functionals in $E_1^*$, $u^*$ and $u_1^*$, then we write
\begin{equation} \label{2_eq_a}
 (1-|z_0|^2) u^* ((Tf)'(z_0)) = (1-|w|^2) v^*(f'(w))\end{equation}
and
 \begin{equation} \label{2_eq_b}(1-|z_0|^2) u_1^* ((Tf)'(z_0)) = (1-|w_1|^2) v_1^*(f'(w_1)). \end{equation}
If $w \neq w_1$ we choose $f_0 \in \mathcal{B}_0(\triangle, E)$ of norm 1 such that
\[ f_0'(w)=\frac{1-|z_0|^2}{1-|w|^2} v \,\,\,\mbox{and} \,\,\, f_0'(w_1)=\frac{1-|z_0|^2}{1-|w_1|^2} v_1.\]

The equations displayed in (\ref{2_eq_a}) and (\ref{2_eq_b}) applied to $f_0$ imply that $$u^*[(Tf_0)'(z_0)]=u_1^* [(Tf_0)'(z_0)]=1.$$
Again the smoothness of $E_1^*$ implies that  $u^*=u_1^*$.  Therefore
\begin{equation} \label{simp_eq}  (1-|w|^2)v^* (f'(w))= (1-|w_1|^2)v_1^*(f'(w_1)), \,\,\, \mbox{for every} \, \, f\in \mathcal{B}_0(\triangle, E).\end{equation}
Since $w \neq w_1$, we select $f$ such that $f'(w)=v$ and $f'(w_1)=v_1$, and (\ref{simp_eq}) applied to this new function implies that $|w|=|w_1|.$ Thus
\[v^* (f'(w))= v_1^*(f'(w_1)), \,\,\, \mbox{for every} \, \, f\in \mathcal{B}_0(\triangle, E).\]

This implies that $v=v_1$. If $w \neq w_1$  we can select a function in $\mathcal{B}_0(\triangle, E)$ such that $f'(w)=v$ and $v_1^*(f'(w_1))\neq 1$, which leads to an absurd. This shows that $w=w_1$ and $\sigma$ only depends on the value of $w$. Therefore $\sigma$ only depends on the value of the first coordinate so it induces a map ( also denoted by $\sigma$) on the open disc. Since $T$ is a surjective isometry the same reasoning applied to the inverse implies that $\sigma$ is  bijective.

We now show that $\sigma$ is analytic. We apply the equation (\ref{meq}) to the functions $f_0(z)= \frac{z^2}{2} v$ and $f_1(z)= zv$ to obtain the following:
\[ (1-|z|^2) u^*[(Tf_0)'(z)]= (1-|\sigma(z)|^2) v^*(f_0'(\sigma(z)))\]
and
\[ (1-|z|^2) u^*[(Tf_1)'(z)]= (1-|\sigma(z)|^2) .\]
For every $z \in \triangle, $ we have $u^*[(Tf_1)'(z)]\neq 0.$ Therefore
\[ \sigma(z) = \frac{u^*[(Tf_0)'(z)]}{u^*[(Tf_1)'(z)]}.\] This shows that $\sigma $ is analytic and then a disc automorphism.
\end{proof}

  \begin{lemma}
If $u^* \in E_1^*$, then  $\Gamma$ restricted to the set  $\{ (z,u^*): z \in \triangle \}$ is constant.
  \end{lemma}
  \begin{proof}
  The equation displayed in (\ref{simp_eq}) is rewritten as
  \[ (1-|z|^2) \, u^* [(Tf)'(z)]= (1-|\sigma (z)|^2) \Gamma (u^*,z) [f'(\sigma(z)], \,\,\,\forall \,\, f \in \mathcal{B}_0(\triangle,E) \,\,\mbox{and} \,\, z \in \triangle.\]
   Therefore we get
  \[ u^* [(Tf)'(z)]=\frac{|\sigma'(z)|}{\sigma'(z)} \Gamma (u^*,z) [(f\circ \sigma)'(z)], \,\, \forall \,\,f\in \mathcal{B}_0(\triangle,E).\]
  Equivalently we write
  \[ \frac{u^* [(Tf)'(z)]}{\Gamma (u^*,z) [(f\circ \sigma)'(z)]}=\frac{|\sigma'(z)|}{\sigma'(z)} .\]
  Thus the left hand side is independent of the choice of $u^*$ and $f$.
  Further,  $\frac{|\sigma'(z)|}{\sigma'(z)}$ is analytic on the open disc because $z  \rightarrow \frac{u^* [(Tf)'(z)]}{\Gamma (u^*,z) [(f\circ \sigma)'(z)]}$  is analytic. An application of  the Maximum Modulus Principle asserts that $\frac{|\sigma'(z)|}{\sigma'(z)}$ is constant, i.e. $\frac{|\sigma'(z)|}{\sigma'(z)}=e^{i\alpha},$ for every $z $ in the disc.

  Then \begin{equation} \label{neqn} u^*[(Tf)'(z)]=e^{i\alpha}\, \Gamma  (u^*,z)[(f\circ \sigma)'(z)].\end{equation}

  We assume that there exist $z_1$ and $z_2$, different points in $\triangle$ and $u^* \in E_1^*$ such that $\Gamma (u^*, z_1)\neq \Gamma (u^*, z_2).$ We consider $f \in \mathcal{B}_0(\triangle, E)$ such that $(f\circ \sigma)'(z_1) =e^{-i\alpha}v_1$ and $(f\circ \sigma)'(z_2) =e^{-i\alpha}v_2,$ with $v_1$ and $v_2$ vectors of norm 1 such that
  \[ \Gamma (u^*, z_1)v_1 = \Gamma (u^*, z_2)v_2=1.\] Substituting  this function in (\ref{neqn}) we obtain $u^* ((Tf)'(z_1))=1$ and $u^*  ((Tf)'(z_2))=1,$  then $(Tf)'(z_1)= (Tf)'(z_2).$

   Therefore
  \[
   e^{i \alpha} \Gamma (u^*,z_1) [(f\circ \sigma)'(z_1)]= e^{i \alpha}\Gamma (u^*,z_2) [(f\circ \sigma)'(z_2)],
  \]
  and $ \Gamma (u^*,z_2) [(f\circ \sigma)'(z_2)]=\Gamma (u^*,z_1) [(f\circ \sigma)'(z_1)], \, \, \forall \, f \in \mathcal{B}_0(\triangle,E).$ This leads to a contradiction and shows that no such pair of points exists. Thus $\Gamma$ only depends  on $u^*$, $\Gamma: E_1^* \rightarrow E_1^*$ and $v^*=\Gamma (u^*)$.
  \end{proof}

  \begin{remark}
  The previous lemma implies that $\Gamma$ induces a mapping from $E^*_1$ onto $E^*_1$, for simplicity also denoted by $\,\Gamma$.
  \end{remark}
  We  collect some useful  properties of $\Gamma$. First $\Gamma (\lambda u^*)=\lambda \Gamma (u^*)$, with $\lambda$  a modulus 1 complex number. Then, for every scalar $\lambda,$ we set $\Gamma (\lambda u^*)=\lambda \Gamma (u^*).$
  In particular,  \[\Gamma \left(\frac{u_1^*+u_2^*}{\|u_1^*+u_2^*\|}\right)=\frac{1}{\|u_1^*+u_2^*\|} \left( \Gamma (u_1^*) +\Gamma (u_2^*) \right).\] Hence, we  extend $\Gamma$ to a linear map $\Gamma : E^* \rightarrow E^*$.
We notice that given two distinct functionals $u_1^*$ and $u_2^*$ we set $\Gamma \left( \frac{u_1^*-u_2^*}{\|u_1^*-u_2^*\|} \right) = w^*$. Therefore $\Gamma (u_1^*) -\Gamma (u_2^*) = \|u_1^*-u_2^*\| w^*$ and
\[ \|\Gamma (u_1^*) -\Gamma (u_2^*)\|\leq \|u_1^*-u_2^*\| .\]
As in \cite{bo_ja_zh} (see pg. 60) we employ the following result due to G. Ding.
\begin{theorem} (see \cite{li_li})\label{li_li} Let $E$ and $F$ be two real Banach spaces. Suppose $V_0$ is a Lipschitz mapping from $E_1$ into $F_1$ (the respective unit spheres) with Lipschitz constant equal to 1, that is $\|V_0(x)-V_0(y)\|\leq \|x-y\|$, for every $x,y$ in $E_1$. Assume also that $V_0$ is a surjective mapping such that for any $x,y \in E_1$ and $r >0$, we have
\[   \|V_0(x)-rV_0(y)\| \wedge \|V_0(x) + r V_0(-y)\|\leq \|x-ry\|\] and $\|V_0(x)-V_0(-x)\|=2.$ Then $V_0$ can be extended to be a real linear isometry from $E$ onto $F$.
\end{theorem}
Since $\Gamma$ satisfies the conditions set in the Theorem \ref{li_li}, this assures the existence  of a surjective real linear isometry from $E^* \rightarrow E^*$ that extends $\Gamma$. For simplicity of notation, we  denote this extension also by $\Gamma$. We observe that the  complex linearity of $T$ implies that of $\Gamma$. Since $E$ is reflexive then the adjoint of $\Gamma$ induces a surjective linear isometry on $E$, we call this isometry $S$, therefore we have
\[ u^* (Tf)'(z)= u^* \left( S (f\circ\sigma)' (z) \right),\]
for every $u^* \in E^*$, $f \in \mathcal{B}_0(\triangle, E)$ and $z \in \triangle.$
This implies that $(Tf)'(z)=  S (f\circ\sigma)' (z). $ A straightforward integration yields
\[ Tf(z)= S[ (f\circ\sigma)(z)-(f\circ\sigma)(0)], \,\, \forall f\in \mathcal{B}_0(\triangle, E)\,,\,\mbox{ and } z \in \triangle.\]

We  summarize these considerations  in the following theorem.
\begin{theorem} \label{mt} Let $E$ be a smooth and strictly convex complex Banach space. Then $T: \mathcal{B}_0(\triangle, E) \rightarrow \mathcal{B}_0(\triangle, E)$ is a surjective linear isometry if and only if there exist a surjective linear isometry $S:E \rightarrow E$ and  a disc automorphism $\sigma $ such that for every $f \in  \mathcal{B}_0(\triangle, E) $ and $z \in \triangle,$
\[ Tf(z)= S[ (f\circ\sigma)(z)-(f\circ\sigma)(0)].\]
\end{theorem}
\begin{proof}
The necessity follows from previous considerations. We now show the sufficiency, i.e. any mapping of the form described in the theorem is indeed a surjective isometry. Such an  operator  is bijective, with inverse  $T^{-1} f (z)= S^{-1} \left[f(\sigma^{-1}(z))-f(\sigma^{-1}(0))\right].$ We now show that $Tf(x)= S[ (f\circ\sigma)(x)-(f\circ\sigma)(0)],$ with $\sigma$ a disc automorphism and $S$ a surjective isometry on $E,$ is an isometry. We have
\begin{align*} \|Tf\|_{\mathcal{B}_0(\triangle, E)}&= sup_{z \in \triangle} (1-|z|^2) \|\sigma'(z) S(f'(\sigma(z)))\|\\ &=sup_{z \in \triangle} (1-|z|^2)|\sigma'(z)|\|f'(\sigma(z))\|. \end{align*}
We set   $w=\sigma (z),$ then if $\sigma (z)= \lambda \frac{z-a}{1- \overline{a}z}$ we have $\sigma^{-1}(w)= \frac{\lambda a+w}{\lambda +\overline{a}w}$. Therefore
\begin{align*} (1-|z|^2) |\sigma'(z)| &= \frac{(1-|a|^2)}{\left| 1- \overline{a}\frac{w+\lambda a}{\lambda +\overline{a}w}\right|^2} \left(1-\left| \frac{w+\lambda a}{\lambda +\overline{a}w}\right|^2\right)\\
& \\
&= (1-|w|^2). \end{align*}
This implies that $\|Tf\|_{\mathcal{B}_0(\triangle, E)}=\|f\|_{\mathcal{B}_0(\triangle, E)}$ and completes the proof.
\end{proof}

\section{Hermitian operators }
In this section we use the form of the surjective isometries to derive information about  the hermitian operators on $\mathcal{B}_0(\triangle, E)$. An operator $A$ is hermitian if and only if $i A$ is the generator of  a strongly continuous one-parameter group of surjective isometries, see \cite{en}. We recall that bounded hermitian operators give rise to uniformly continuous one-parameter groups of surjective isometries.

We consider a family of one-parameter group of surjective isometries on ${\mathcal{B}_0(\triangle, E)}$,  Theorem \ref{mt} implies that each isometry determines both a disc automorphism  and a  surjective isometry on $E.$
The next proposition states that the group properties of the underlying group of  isometries transfer to the defining families.
\begin{proposition}\label{form_of_s_i} Let $E$ be a smooth and strictly convex complex Banach space, then
 $\{T_t\}_{t\in \mathbb{R}}$ is a one parameter group of surjective isometries on $\mathcal{B}_0(\triangle, E)$ if and only if there exist a one parameter group of disc automorphisms $\{\sigma_t\}_{t\in \mathbb{R}}$ and one parameter group of surjective isometries on $E$, $\{S_t\}_{t\in \mathbb{R}}$ such that
 \[ T_t (f)(z)= S_t [ f(\sigma_t(z))-f(\sigma_t(0))], \,\,\, \forall \,\, f \in \mathcal{B}_0(\triangle, E).\]
\end{proposition}
\begin{proof}
Let  $\{T_t\}_{t\in \mathbb{R}}$ be a one parameter group of surjective isometries on $\mathcal{B}_0(\triangle, E).$ If $T_0=Id$ we have
\[ S_0[f\circ \sigma_0 - f(\sigma_0(0))]=f , \,\,\, \forall \,\, f \in \mathcal{B}_0(\triangle, E).\] For $f_1 (z)=z v$ and $f_2(z)= z^2 v,$ with $v$ a unit vector in $E$, we obtain
\begin{align*}
[\sigma_0(z)-\sigma_0(0)]\, S_0(v)&=z v\\
[\sigma_0(z)^2-\sigma_0(0)^2]\, S_0(v)&=z^2 v.
\end{align*}
This implies that $ [\sigma_0(z)+\sigma_0(0)]\ \, z v=z^2 v $ and $\sigma_0(z)+\sigma_0(0)=z,$ for every $z \in \triangle$. If $z=0$ then $\sigma_0 (0)=0$  and $\sigma_0 (z)=z.$
Given  $t$ and $s$ in $\mathbb{R}$, we have  $T_{t+s}(f)=T_t[T_s(f)]$, then
\begin{align*}
T_t[T_s(f)]&= S_t [T_s(f)\circ \sigma_t - T_s(f) (\sigma_t(0))]\\
&= S_t \{ S_s [f(\sigma_s\circ \sigma_t) - f (\sigma_s (0))]-S_s [ f(\sigma_s\circ \sigma_t)(0)-f(\sigma_s(0))]\}\\
&= S_t S_s (f(\sigma_s\circ \sigma_t)-f(\sigma_s(\sigma_t(0)))).
\end{align*}
On the other hand, $T_{t+s}(f)=S_{t+s}[f\circ \sigma_{t+s}- f(\sigma_{t+s}(0))].$ Hence
\[(*) \,\,\, S_{t+s}[f\circ \sigma_{t+s}- f(\sigma_{t+s}(0))]= S_t S_s (f(\sigma_s\circ \sigma_t)-f(\sigma_s(\sigma_t(0)))), \,\,\, \forall \, f\in \mathcal{B}_0(\triangle, E).\]
In particular, for $f_1$ and $f_2$  defined above, we have
\begin{align*}
[S_tS_s v] [(\sigma_s\circ \sigma_t)(z)-(\sigma_s\circ \sigma_t)(0)]&=S_{t+s}v [\sigma_{s+t}(z)-\sigma_{t+s}(0)]\\
[S_tS_s v] [(\sigma_s\circ \sigma_t)(z)^2-(\sigma_s\circ \sigma_t)(0)^2]&=S_{t+s}v [\sigma_{s+t}(z)^2-\sigma_{t+s}(0)^2].
\end{align*}
Therefore \[[(\sigma_s\circ \sigma_t)(z)+(\sigma_s\circ \sigma_t)(0)][\sigma_{s+t}(z)-\sigma_{t+s}(0)]=\sigma_{s+t}(z)^2-\sigma_{t+s}(0)^2.\]
This implies
\[(\sigma_s\circ \sigma_t)(z)+(\sigma_s\circ \sigma_t)(0)=\sigma_{s+t}(z)+\sigma_{t+s}(0), \,\,\, \forall z \in \triangle.\] For $z=0$ we have $(\sigma_s\circ \sigma_t)(0)=\sigma_{t+s}(0).$ Then $\sigma_s\circ \sigma_t=\sigma_{s+t}$ and  from  (*) we conclude that $S_t S_s=S_{t+s}.$  The converse implication follows from straightforward calculations. This concludes the proof.
\end{proof}
The next result addresses the question of whether the strong  continuity of a one-parameter group of surjective isometries $\{T_t\}_{t \in \mathbb{R}}$  also transfers to the defining symbols.
\begin{proposition} \label{one_para_se_gr} Let $E$ be a smooth and strictly convex complex Banach space.
If  $\{T_t\}_{t \in \mathbb{R}}$ is a strongly continuous one parameter group of surjective isometries on $ \mathcal{B}_0(\triangle, E)$, then there there exist a strongly continuous one parameter group of surjective isometries on $E$, $\{S_t\}_{t\in \mathbb{R}}$ and a continuous one parameter group of disc  automorphisms $\{\sigma_t\}_{t\in \mathbb{R}}$ such that
\[ T_t (f) (z)= S_t ( f(\sigma_t (z)) - f(\sigma_t(0))), \,\,\,\forall f \in \mathcal{B}_0(\triangle, E) \,\,\, \forall z \in \triangle.\]
\end{proposition}
\begin{proof}
Proposition \ref{form_of_s_i} implies the existence of one parameter groups   of surjective isometries on $E$ and disc automorphisms, $\{S_t\}$ and $\{\sigma_t\}$ respectively, such that
\[  T_t (f) (z)= S_t ( f(\sigma_t (z)) - f(\sigma_t(0))), \,\,\,\forall f \in \mathcal{B}_0(\triangle, E) \,\,\, \forall z \in \triangle.\]
Since  $\{T_t\}_{t \in \mathbb{R}}$ is strongly continuous, in particular for $f_1(z) = z \mathbf{v}$, $f_2(z)= z^2 \mathbf{v}$ and $f_3(z)= z^3 \mathbf{v}$ ( $\mathbf{v}\in E_1$, $z \in \triangle$ and $i=1,2,$ or $3$) we have
\[ \|[\sigma_t (z)^i -\sigma_t(0)^i ] S_t(\mathbf{v}) - z^i \mathbf{v}\|\rightarrow \, 0  \,\,\,\mbox{as} \,\,\, t \rightarrow 0.\]
Given  $z_0 \neq 0$, and $\varphi \in E_1^*$ such that $\varphi (\mathbf{v})=1$,
\[ \lim_{t \rightarrow 0} [\sigma_t (z_0) -\sigma_t(0)]\varphi (S_t(\mathbf{v})) = z_0 \,\,\,\,\mbox{and} \,\,\, \lim_{t \rightarrow 0} [\sigma_t (z_0)^2 -\sigma_t(0)^2]\varphi (S_t(\mathbf{v})) = z_0^2,\]
implies that \begin{equation} \label{cont1} \lim_{t \rightarrow 0} (\sigma_t (z_0) +\sigma_t(0))=z_0.\end{equation}  Also \[ \lim_{t \rightarrow 0} [\sigma_t (z_0) -\sigma_t(0)]\varphi (S_t(\mathbf{v})) = z_0 \,\,\,\,\mbox{and} \,\,\, \lim_{t \rightarrow 0} [\sigma_t (z_0)^3 -\sigma_t(0)^3]\varphi (S_t(\mathbf{v})) = z_0^3,\]
implies \begin{equation} \label{cont2}\lim_{t \rightarrow 0} (\sigma_t (z_0)^2 + \sigma_t(z_0)\sigma_t(0) +\sigma_t(0)^2)=z_0^2.\end{equation} It follows from (\ref{cont1}) and (\ref{cont2}) that  $\lim_{t \rightarrow 0} \sigma_t(z_0)\sigma_t(0)=0.$ This implies that $\lim_{t \rightarrow 0} \sigma_t(0)=0$, otherwise there exists a sequence $\{t_n\}$ such that $\sigma_{t_n}(0)$ would converges to some complex number $w(\neq 0)$ in the closed disc. Hence, for every $z_0 \neq 0$ $\{ \sigma_{t_n}(z_0)\}_n$ converges to zero and $w=z_0$. This leads to an absurd and proves  that $\lim_{t \rightarrow 0} \sigma_t(0)=0$ and $\lim_{t \rightarrow 0} \sigma_t(z_0)=z_0.$ This establishes the continuity of $\{\sigma_t\}$. For  $z_0 \neq 0$,
\[ \lim_{t \rightarrow 0} \frac{ [\varphi_t(z_0)-\varphi_t(0)] S_t (\mathbf{v})}{ \varphi_t(z_0)-\varphi_t(0)} = \frac{z_0 \mathbf{v}}{z_0}=\mathbf{v},\] which completes the proof.
\end{proof}

\begin{corollary}\label{nc}
Let $E$ be a smooth and strictly convex complex Banach space. If $A$ is a (not necessarily bounded) hermitian operator on $\mathcal{B}_0(\triangle, E)$, then there exist a hermitian operator (not necessarily bounded) $V$ on $E$ and a continuous group of disc automorphisms $\{\sigma_t\}_{t \in \mathbb{R}}$ such that
\[ A(f)(z) = V [f(z)] + [\partial_t \, \sigma_t(z)]_{t=0} f'(z).\]
If $A$ is bounded then $\{\sigma_t\}_{t \in \mathbb{R}}$ is the trivial group and $ A(f)(z) = V [f(z)],$ with $V$ bounded.
\end{corollary}

Nontrivial disc automorphisms can be extended to  conformal maps on the plane and as such, they are characterized according to their fixed points. More precisely, they fall into three types: an  elliptic automorphism has a single fixed point in the disc and another one in the interior of  its complement; a hyperbolic automorphism has two distinct fixed points on the boundary of the disc and  a parabolic has a single fixed point on the boundary of the disc, cf. \cite{krantz}.

It has been shown that all disc automorphisms of a nontrivial one-parameter group family of disc automorphisms share the same fixed points, cf. \cite{be_po1}. Thus, we consider the following three cases:
\begin{enumerate}
\item[(i)] Elliptic.\[  \varphi_t (z) = \frac{ ( e^{ict}-|\tau|^2)z - \tau (e^{ict}-1)}{ 1-|\tau|^2 e^{ict} - \bar{\tau} (1- e^{ict})z}, \] with $c \in \mathbb{R}\setminus \{0\}$, $ \tau \in \mathbb{C}$ such that $|\tau|<1.$
    \item[(ii)] Hyperbolic.
    \[ \varphi_t (z) = \frac{(\beta e^{c t }-\alpha)z + \alpha \beta (1- e^{c t})}{ ( e^{c t}-1)z + (\beta - \alpha e^{c t})},\]
    with $c$ a positive real number, $|\alpha|=| \beta|=1$ and $\alpha \neq \beta.$
 \item[(iii)] Parabolic.
    \[ \varphi_t(z) = \frac{(1-ict) z + ict \alpha}{-ic \bar{\alpha} t z +1+ i c t},\] with $c \in R \setminus\{0\}$ and $|\alpha|=1.$
\end{enumerate}

In \cite{be_ka_po}, Berkson, Kaufman and Porta show the existence of an invariant polynomial associated with one parameter group of disc automorphisms \[ \varphi_t(z) = a(t) \frac{z-b(t)}{1- \overline{b(t)} \, z},\] with $|a(t)|=1$ and $|b(t)|<1.$ This polynomial is given by
\[ P(z)= \overline{b'(0)} z^2+ a'(0)z - b'(0).\] It is a straightforward computation to check that
\[ \partial_t\varphi_{t}(z)|_{t=0}= P(z) \,\,\, \mbox{and} \,\,\,\partial_t \varphi'_{t} (z)|_{t=0} =P'(z).\]

The invariant polynomial for each of the three types of nontrivial disc automorphisms is given by:
\begin{enumerate}
\item[(i)] Elliptic. $P(z)= -\frac{ic}{1-|\tau|^2} \left\{ (\overline{\tau} z-1)(z-\tau)\right\}$  ($|\tau|<1$).
 \item[(ii)] Hyperbolic. $P(z)=-\frac{c}{\beta-\alpha} \left\{ z^2 - (\alpha+\beta) z + \alpha \beta\right\}, $ ($|\alpha|=| \beta|=1$ and $\alpha \neq \beta$).
  \item[(iii)] Parabolic. $P(z)=i \, \overline{\alpha}\, c (z-\alpha)^2,$  ($c \in R \setminus\{0\}$ and $|\alpha|=1$).
\end{enumerate}
Since hermitian operators  are generators of strongly continuous one-parameter groups of surjective isometries we  derive a representation for the  $\mathcal{B}_0(\triangle, E)$ setting.
\begin{proposition} \label{herm}
Let $E$ be a smooth and strictly convex complex Banach space. If a  closed  operator $A$ with domain $\mathcal{D}(A)$, a dense subset of  $\mathcal{B}_0(\triangle, E)$ is hermitian then  there exists a closed and densely defined hermitian operator $V$ on $E$ and a nonzero real number $c$, and complex numbers $\tau$, $\alpha$ and $\beta$ such that $|\tau|<1$ and $|\alpha|=| \beta|=1$  and one of the following holds:
\begin{enumerate}
\item $A(f)(z)= V (f(z)),\,\, f \in \mathcal{B}_0(\triangle, E)$ and $z \in \triangle.$
\item $A(f)(z)= V (f(z))  +\frac{c}{1-|\tau|^2} \left\{ (\overline{\tau} z-1)(z-\tau)\right\} f'(z), \,\, f \in \mathcal{D}(A)$ and $z \in \triangle.$
    \item $A(f)(z)= V (f(z)) -i\frac{|c|}{\beta-\alpha} \left\{ z^2 - (\alpha+\beta) z + \alpha \beta\right\}f'(z), \,\, f \in \mathcal{D}(A)$ and $z \in \triangle.$
        \item $A(f)(z)= V (f(z)) - \, \overline{\alpha}\, c (z-\alpha)^2f'(z), \,\, f \in \mathcal{D}(A)$ and $z \in \triangle.$
\end{enumerate}
\end{proposition}
\begin{proof}
Given a hermitian operator $A$ satisfying the conditions stated, then $\{e^{-itA}\}_{t \in \mathbb{R}}$ is a strongly continuous one-parameter group of surjective isometries on $\mathcal{B}_0(\triangle, E)$. Theorem \ref{mt} applies to assert the existence of a strongly continuous one-parameter group of surjective isometries on $E$, $\{S_t\}_{t\in \mathbb{R}}$ and a continuous group of disc automorphisms $\{\sigma_t\}_{t \in \mathbb{R}}$ such that
\[ e^{-itA} (f)(z)= S_t (f(\sigma_t(z))-f(\sigma_t(0))), \,\,\, \forall \,\,\, f \in \mathcal{D}(A).\]
We denote by $V$ the generator of $\{S_t\}_{t\in \mathbb{R}}$ then
\[ A (f)(z)= V(f(z)) -i \partial_t(\sigma'_t(z))|_{t=0} \, f'(z), \,\, \forall \,\,\, f \in \mathcal{D}(A).\]  The considerations in the preamble to the proposition justify the three cases listed.

\end{proof}

\begin{remark}
In the scalar case, $\mathcal{B}(\triangle)$ is known be a Grothendieck space with the Dunford Pettis property (see \cite{lo}). As a consequence of this fact Blasco et. al. in \cite{bl} showed that all strongly continuous groups on $\mathcal{B}(\triangle)$ are uniformly continuous. Therefore only the trivial group of disc automorphisms is permissible (i.e. $\{\sigma_t\} =\{ id\}$) and the hermitian operators are just real multiples of the identity. This is in contrast with our case because of the following example. Suppose $E= \ell_2$, $\sigma_t(z)=z$ and set
\[T_t(f)(z)= (e^{it} f_1(z), e^{2it} f_2(z), \ldots ).\] This is a family of strongly continuous  surjective isometries but not uniformly continuous. The generator of this group is given by
\[ Af (z)= (f_1(z), 2f_2(z), 3f_3(z) \ldots )\]
which is clearly an unbounded operator.
\end{remark}
We also have the following characterization for bounded hermitian operators on $\mathcal{B}_0(\triangle, E)$.
\begin{corollary}\label{bp} Let $E$ be a smooth and strictly convex complex Banach space. If $A$ is a bounded hermitian operator on $\mathcal{B}_0(\triangle, E)$ then there exists a bounded hermitian operator $V$ on $E$ such that
\[ A(f)(z)=V(f(z)), \,\,\, \forall \, f \in \mathcal{B}_0(\triangle, E) \,\,\, \mbox{and} \,\,\, z \in \triangle.\]
\end{corollary}
\begin{proof}
The operator $A$ is of one of the forms listed in the Proposition \ref{herm}, the sequence of functions $f_n(z)=z^n \mathbf{v},$ with $\mathbf{v}$ a unit vector in $E,$ are in $\mathcal{B}_0(\triangle, E)$. Thus the respective sequence of norms is uniformly bounded and $\|Af\|$ is unbounded. This implies that $\sigma'_t(z)|_{t=0}=0$ and $\sigma_t(z)=z.$ This completes the proof.
\end{proof}
\begin{remark}
It is a known fact that Banach spaces with the Grothendieck  property and the Dunford-Pettits property only support bounded hermitian operators, see \cite{lo,bl}. The little Bloch scalar valued space, $\mathcal{B}_0(\triangle )$ has these two properties (cf. \cite{lo}) and thus every hermitian operator on $\mathcal{B}(\triangle )$ is bounded. This implies that if a hermitian operator $A$  on $\mathcal{B}_0(\triangle,E )$ with an eigenspace containing  one dimensional subspace $\{h(z) v: h \in \mathcal{B}(\triangle ), \, v \in E_1\}$ then $A$ is of the form $A(f)(z)=V f(z)$.
\end{remark}

 Corollary \ref{bp} allows us to extend our characterization to surjective isometries of $\mathcal{B}_*(\triangle, E)$.
As pointed out in Remark \ref{l1_sum}, $\mathcal{B}_*(\triangle, E)$ is isometrically isomorphic to the $\ell_1$-sum of $E$ with $\mathcal{B}_0(\triangle, E)$. Moreover, if $E$ does not admit $L_1$-projections (i.e. a  bounded hermitian operator  $P$ on $E$ such that $P^2=P$ and for every $v \in E,$  $\|v\|_E= \|Pv\|_E+ \|(I-P)v\|_E$) then also $\mathcal{B}_0(\triangle, E)$ does not admit $L_1$-projections. In fact, assuming $P$ represents a $L_1$-projection on $\mathcal{B}_0(\triangle, E)$,  Corollary \ref{bp} implies that  $P(f)(z)= V (f(z))$, with $V$ a bounded hermitian projection on $E$. Therefore $P (h \mathbf{v})(z)=h(z) V\mathbf{v}, $ for $h \in \mathcal{B}_0(\triangle)$. In particular for $h(z)=z,$  $\|\mathbf{v}\| = \|V\mathbf{v}\|+ \|(I-V)\mathbf{v}\|$ which implies that $E$ supports  $L_1$-projections.

 We employ Proposition 4.3 in \cite{jj},  a surjective isometry on $\mathcal{B}_*(\triangle, E)$ can be written as a direct sum of a surjective  isometry on $E$ and a surjective isometry on $\mathcal{B}_0(\triangle, E)$. Therefore,  a surjective isometry $T$ on $\mathcal{B}_*(\triangle, E)$ is given by
\[ T(f)(z)= Uf(0)+ S[ (f\circ\sigma)(x)-(f\circ\sigma)(0)],\] with $\sigma$ a disc automorphism, $U$ and  $S$  surjective isometries on $E.$
We summarize these considerations in the next result.
\begin{theorem}\label{form_iso_in*sp}
Let $E$ be a smooth, strictly convex complex Banach space. Then $T:\mathcal{B}_*(\triangle, E) \rightarrow \mathcal{B}_*(\triangle, E)$ is a surjective linear isometry if and only if there exist surjective linear isometries on $E$, $U$ and $V$,  and a disc automorphism $\sigma$ such that for every $f \in \mathcal{B}_*(\triangle, E)$ and $ z \in \triangle$,
\[ Tf(z)= U[f(0)]+ V [ f(\sigma (z))- f(\sigma(0))].\]
\end{theorem}
The next corollary extends the results stated in Propositions \ref{form_of_s_i} and \ref{one_para_se_gr} to $\mathcal{B}_*(\triangle, E)$.
\begin{corollary} Let $E$ be a smooth, strictly convex complex Banach space.
Then  $\{T_t\}_{t \in \mathbb{R}}$ is a strongly continuous one parameter group of surjective isometries on $ \mathcal{B}_*(\triangle, E)$ if and only if there exist a continuous one parameter group of disc  automorphisms $\{\sigma_t\}_{t\in \mathbb{R}}$ and  strongly continuous one parameter groups of surjective isometries on $E$, $\{U_t\}_{t\in \mathbb{R}}$ and $\{S_t\}_{t\in \mathbb{R}}$   such that
\[ T_t (f) (z)= U_t(f(0))+S_t ( f(\sigma_t (z)) - f(\sigma_t(0))), \,\,\,\forall f \in \mathcal{B}_0(\triangle, E) \,\,\, \forall z \in \triangle.\]
\end{corollary}
\begin{proof}
Since  $E$ is a smooth, strictly convex complex Banach space, it does not support any $L_1$-projection, Theorem \ref{form_iso_in*sp} applies and for each $t \in \mathbb{R},$
\[ T_t (f) (z)= U_t(f(0))+S_t ( f(\sigma_t (z)) - f(\sigma_t(0))), \,\,\,\forall f \in \mathcal{B}_0(\triangle, E) \,\,\, \forall z \in \triangle.\]
The proof given for Proposition \ref{one_para_se_gr} shows that $\{\sigma_t\}_{t\in \mathbb{R}}$ is a one parameter group of disc automorphisms and $\{S_t\}_{t\in \mathbb{R}}$  is a strongly continuous one parameter group of surjective isometries on $ E$. Then by considering constant functions we also derive that $\{U_t\}_{t\in \mathbb{R}}$  is  a strongly continuous one parameter group of surjective isometries on $ E$. The converse implies follows from straightforward computations.
\end{proof}

\begin{corollary}\label{hr_bloch}
Let $E$ be a Hilbert space. If $A$ is a (not necessarily bounded) hermitian operator on $\mathcal{B}_*(\triangle, E)$, then there exist hermitian operators (not necessarily bounded) $U$ and $V$ on $E$ and a continuous group of disc automorphisms $\{\sigma_t\}_{t \in \mathbb{R}}$ such that
\[ A(f)(z) =U[f(0)]+ V [f(z)] + [\partial_t \, \sigma_t(z)]_{t=0} f'(z).\]
If $A$ is bounded then  $ A(f)(z) =U[f(0)]+ V [f(z)],$ with $U$ and $V$ bounded.
\end{corollary}

\section{Generalized bi-circular projections }
In this section we characterize the generalized bi-circular projections on $\mathcal{B}_0(\triangle, E)$. We recall that a generalized bi-circular projection $P$ satisfies   $P^2=P$ and $P+\lambda (I-P)=T$ with $T$ a surjective isometry and $\lambda $ a modulus 1 complex number different from 1, \cite{fosner}. We refer the reader to the following papers for additional information about this type of projections, \cite{bo_ja}.

A straightforward computation yields the following algebraic equation  $T^2-(\lambda +1) T+ \lambda I=0.$

\begin{theorem}\label{gbp}
Let $E$ be a smooth and strictly convex complex  Banach space. Then $P$ is a generalized bi-circular projection on $\mathcal{B}_0(\triangle, E)$ if and only if there exists a reflexive isometry $T$ (i.e. $T^2=I$) such that
\[ P= \frac{I+T}{2}.\]
\end{theorem}
\begin{proof}
If $P$ is a generalized bi-circular projection then $P+\lambda (I-P)=T$ with $\lambda\in \mathbb{T}\setminus \{1\}$ and $T$ a surjective isometry. An application of Theorem \ref{mt} implies that  there exist a surjective linear isometry $V:E \rightarrow E$ and  a disc automorphism $\sigma $ such that for every $f \in  \mathcal{B}_0(\triangle, E) $ and $z \in \triangle$
\[ Tf(z)=V[ (f\circ\sigma)(z)-(f\circ\sigma)(0)].\] The automorphism $\sigma$ is of the form $\sigma(z)= \mu \frac{z-\alpha}{1-\overline{\alpha} z}$ with $\mu \in \mathbb{T}$ and $|\alpha |<1.$
The condition $P^2=P$ implies that $T^2-(\lambda +1) T+ \lambda I=0.$ Therefore we have
\begin{equation} \label{eq_for_V} V^2 [ f((\sigma \circ \sigma)(z))-f((\sigma \circ \sigma)(0))] -(\lambda +1) V [f(( \sigma)(z))-f(( \sigma)(0))]+\lambda f(z)=0, \end{equation} for every $ f \in \mathcal{B}_0(\triangle, E)$ and $z\in \triangle.$
By differentiating (\ref{eq_for_V}) we obtain
\begin{equation} \label{deq} V^2 [ f'((\sigma \circ \sigma)(z)) \sigma'(\sigma(z)) \sigma'(z) ]-(\lambda +1) V [f'(( \sigma)(z)) \sigma'(z)] +\lambda f'(z)=0.\end{equation}
The equation displayed in (\ref{deq}) applied to $f(z)= \frac{z^2}{2} \mathbf{v}$ (with $\mathbf{v}$ a vector in $E$ of norm 1) and with $z=\alpha$ yields
\[ V^2 \mathbf{v}= \frac{\lambda}{\mu^3} \mathbf{v}.\]
Applying  (\ref{deq}) to $f(z)= \frac{z^2}{2} \mathbf{v}$ and setting $z=0$ we obtain
\[ (V^2 \mathbf{v}) \, \mu^3 \frac{-\mu \alpha -\alpha}{1+\mu |\alpha|^2} \frac{1- |\alpha|^2}{(1+\mu |\alpha|^2)^2}(1-|\alpha|^2) - ( V \mathbf{v} ) (\lambda +1) (-\mu \alpha) \mu (1-|\alpha|^2) =0.\]
We assume that $\lambda \neq -1$, then straightforward calculations show that
\begin{equation}\label{eq1_for_V}  V= \frac{\lambda (\mu+1)(1-|\alpha|^2)}{(\lambda +1) \mu^2 (1+\mu |\alpha|^2)^3} I.\end{equation}
This last equation implies that $\mu \neq -1$.
Once more, applying equation (\ref{deq}) to $f(z)= z \mathbf{v}$ and setting $z=\alpha$ we obtain
 \begin{equation}\label{eq2_for_V} V = \frac{\lambda (\mu+1)(1-|\alpha|^2)}{\mu^2 (\lambda +1)} I.\end{equation}
  From (\ref{eq1_for_V}) and (\ref{eq2_for_V}) we derive  $(1+\mu |\alpha|^2)^3=1.$ This leads to $1+\mu |\alpha|^2=1$, $1+\mu |\alpha|^2= \cos \frac{\pi}{3} + i \sin \frac{\pi}{3} $ or $1+\mu |\alpha|^2= \cos \frac{2\pi}{3} + i \sin \frac{2\pi}{3}. $ It is easy to show that only the first equation leads to the solution $\alpha =0.$ Therefore  $V= \frac{\lambda (\mu+1)}{ (\lambda +1) \mu^2} I$ and $\sigma (z)= \mu z.$ Since $V$ is an isometry the $|\mu+1|=|\lambda +1|$, and thus $\mu = \lambda $ or $\lambda = \overline{\mu}. $

 We consider two cases.

 1. If $\lambda = \mu$ then $V= \overline{\lambda} I$ and equation (\ref{eq_for_V}) applied to $f(z)= z \mathbf{v}$ implies
 \[ \lambda^4 -\lambda (\lambda+1) +\lambda =0\]
 and thus $\lambda=1.$ This is impossible.

 2. If $\lambda = \overline{\mu}$ then $V= \overline{\mu}^2 I$. We   differentiate equation (\ref{eq_for_V}) and applied to $f(z)= z^3 \mathbf{v}$ to obtain
 \[ \mu^4 - (\mu +1) \mu^2 + \mu=0.\] This equation has solutions $\pm 1$. Either case leads to a contradiction since we have assumed that $\lambda\neq -1.$

 This contradiction shows that $\lambda=-1$ and (\ref{deq}) reduces to
\begin{equation} \label{leq} V^2 [ f'((\sigma \circ \sigma)(z)) \sigma'(\sigma(z)) \sigma'(z) ] = f'(z),\end{equation}  which applied to $f(z)= \frac{z^2}{2} \mathbf{v}$ with $z=\alpha$ yields
\[ (V^2 \mathbf{v} ) (-\mu^3 \alpha) =\alpha \mathbf{v}. \]
Therefore $V^2 = -\overline{\mu}^3 I.$

The equation (\ref{leq}) applied to $f(z)= z \mathbf{v}$  and $z=\alpha$ gives
\[ -\overline{\mu}^3 \sigma'(0) \sigma'(\alpha) = 1.\] Therefore $\mu=-1$ and $V^2=I.$ We also have $\sigma\circ \sigma (z)=z.$ Therefore $T^2=I$ and proves that $P$ is the average of the identity operator with a reflection. The reverse implication is clear.

\end{proof}
A generalized bi-circular projection $P$ on $\mathcal{B}_*(\triangle, E)$ is given
\[ P= \frac{1}{1-\lambda} (T -\lambda I)\]
with $T$ a surjective isometry on $\mathcal{B}_*(\triangle, E)$ and $\lambda$ a modulus $1$ scalar different from $1$. Theorem \ref{form_iso_in*sp} implies the existence of surjective isometries on $E$, $U$ and $V$, also a disc automorphism $\sigma$  such that $T(f) (z)= U (f(0))+ V [ f(\sigma (z))-f(\sigma(0))].$

The form for the surjective isometries  on $\mathcal{B}_*(\triangle, E)$ implies that $P$ leaves invariant the subspace of all constant functions and also $\mathcal{B}_0(\triangle, E)$. Applying Theorem \ref{gbp} we conclude  that the restriction of $P$ to $\mathcal{B}_0(\triangle, E)$ is the average of $I$ with an isometric reflection on $\mathcal{B}_0(\triangle, E)$, thus $V^2=Id$ and $\sigma^2=id_{\triangle}$. Therefore $P$ is the average of the identity on  $\mathcal{B}_*(\triangle, E)$ with a surjective isometry $T$. Since $T=2P-I$ is such that  $T^2=I$,  then  generalized bi-circular projections on $\mathcal{B}_*(\triangle, E)$ are the average of the identity operator with an isometric reflection.


\end{document}